\documentclass[10pt,twoside]{article}
\usepackage{mathrsfs}
\usepackage{amssymb}
\usepackage{amsmath,color}
\usepackage[all]{xy}
\usepackage{amsthm}
\usepackage{url}

\usepackage{tikz}
\usepackage{stmaryrd}
\usepackage{xcolor}
\usetikzlibrary{arrows,calc}
\usepackage{etex}

\numberwithin{equation}{section}

 \def\Hom{\mbox{\rm Hom}} \def\dim{\mbox{\rm dim}\,} 
    \def\mod{\mbox{\rm \textbf{mod}}\,}

\def\cocone{\mbox{\rm cocone}}
\def\B{\mathcal {B}}\def\C{\mathcal {C}}
\def\A{\mathcal{A}} 
\def\Id{\mbox{\rm Id}\,} \def\Im{\mbox{\rm Im}\,}

\def\C{\mathcal{C}}
\def\E{\mathbb{E}}
\def\s{\mathfrak{s}}
\def\Ab{\mathit{Ab}}
\def\del{\delta}

\newcommand{\bsm}{\begin{smallmatrix}}
\newcommand{\esm}{\end{smallmatrix}} 

\newtheorem{theorem}{Theorem}[section]

\newtheorem{definition}[theorem]{Definition}
\newtheorem{remark}[theorem]{Remark}
\newtheorem{lemma}[theorem]{Lemma}
\newtheorem{example}[theorem]{Example}

\newtheorem{corollary}[theorem]{Corollary}
\newtheorem{condition}[theorem]{Condition}
\newtheorem{theorem*}{Theorem}

\newcommand{\add}{\operatorname{add}}

\newcommand{\gl}{\operatorname{gl.dim}}

\newcommand{\pd}{\operatorname{pd}}

\newcommand{\Ker}{\operatorname{Ker}}
\newcommand{\Coker}{\operatorname{Coker}}

\newcommand{\RNum}[1]{\uppercase\expandafter{\romannumeral #1\relax}}

\title{ \bf Homological Dimensions of Extriangulated Categories and Recollements 
}
\vspace{0.2cm}

\author{ Weili Gu$^1$, Xin Ma$^{1,}$\thanks{Corresponding author. maxin0719@126.com.}, Lingling Tan$^2$
\\
{\it \footnotesize 1. College of Science, Henan University of Engineering, Zhengzhou 451191, P.R. China}
\\
{\it \footnotesize 2. School of Mathematical Sciences, Qufu Normal University, Qufu 273165, P.R. China}
\\
}
\date{ }
\begin{document}

\baselineskip=16pt
\maketitle

\begin{abstract}
In this paper, let $(\mathcal{A},\mathcal{B},\mathcal{C})$ be a recollement of extriangulated categories.
We introduce the global dimension and extension dimension of extriangulated categories,
and  give some upper bounds of global dimensions (resp. extension dimensions)
of the categories involved in $(\mathcal{A},\mathcal{B},\mathcal{C})$,
which give a simultaneous generalization of some results in the recollement of abelian categories and triangulated categories.
\end{abstract}

\pagestyle{myheadings}
\markboth{\rightline {\scriptsize W. Gu, X. Ma, L. Tan }}
{\leftline{\scriptsize Homological Dimensions of Extriangulated Categories and Recollements }}



\section{Introduction}
Abelian categories and triangulated categories are two fundamental structures in algebra and geometry.
 Recently, Nakaoka
and Palu \cite{Na} introduced the notion of extriangulated categories which is extracting properties on triangulated categories and exact categories (in particular, abelian categories).
There exist extriangulated categories which are not exact
nor triangulated, many examples can be founded in \cite{HZZ20P,INY18A, Na}.
Recollements of triangulated categories and abelian categories were introduced by Be{\u\i}linson, Bernstein and Deligne \cite{BBD},
which play an
important role in algebraic geometry and representation theory, for instance \cite{BBD,BARI07H,CEPBSL86D,CEPBSL88F,KNJ94G}.
They are closely related each other, and possess similar properties in many aspects.
In order to give a simultaneous generalization of recollements of abelian categories and triangulated categories,
Wang, Wei and Zhang \cite{WWZ20R} introduced the notion of recollements of extriangulated categories.

Homological dimension is an important invariant in representation theory, which plays an important role in the study of triangulated categories and abelian categories
(see for instance \cite{BA08S,B00R,DHL,HYG,PYY,PC14H,RR08D,ZHJ,ZJL}).
In the situation of recollements, Psaroudakis \cite{PC14H} studied the global dimension, finitistic dimension, representation dimension in a recollement of abelian categories;
Hu and Yao \cite{HYG} studied relative homological dimension in a recollement of triangulated categories with respect to proper classes of triangles;
Zheng, Ma and Huang \cite{ZJL} studied the extension dimension in a recollment of abelian categories;
Zhang and Zhu \cite{ZHJ} studied the Gorenstein global dimension in a recollement of abelian categories; etc..

In this paper, we introduce the notions of global dimension and extension dimension in extriangulated categories and investigate their behaviors in a recollement of extriangulated categories.

The paper is organized as follows:

In Section 2, we summarize some basic definitions and properties of extriangulated categories, which will be used in this sequel.

In Section 3, we introduce the notion of global dimension of extriangulated categories,  which gives a simultaneous generalization of $\xi$-global dimension of triangulated categories with respect to a proper class $\xi$ of triangles (see \cite{B00R}) and global dimension of  exact categories.
Later, we give some bounds for global dimension of the categories involved in a recollement of extriangulated categories, which generalizes  some results in \cite{PC14H} and is new in exact categories.

In Section 4, we introduce the notion of extension dimension of extriangulated categories,
 which gives a simultaneous generalization of dimension of triangulated categories and extension dimension of abelian categories.
 The relationship of extension dimension of the categories involved in a recollement of extriangulated categories is given, which unifies the results in triangulated categories \cite{PC14H} and abelian categories \cite{ZJL}.

Finally, in Section 5, we give some examples to illustrate the obtained results.

Throughout this paper,
all subcategories are assumed to be full, additive and closed under isomorphisms.
Let $\mathcal{C}$ be an extriangulated category, and let $\mathcal{U}$ be a class of objects of $\mathcal{C}$. We use $\add \mathcal{U}$ to
denote the subcategory of $\mathcal{C}$ consisting of
direct summands of finite direct sums of objects in $\mathcal{U}$.
Let $\Lambda$ be an artin algebra. We use $\mod \Lambda$ to denote the category of finitely generated left $\Lambda$-modules.
\section{Preliminaries}
Throughout, let $k$ be a field and $\mathcal{C}$ be a Krull-Schmidt Hom-finite, $k$-linear additive category.
We first recall some definitions and some basic properties of extriangulated categories from \cite{Na}.

Suppose that $\C$ is equipped with a biadditive functor $\E\colon\C^{op}\times\C\to\Ab$, where $\Ab$ is the category of abelian groups. For any pair of objects $A,C\in\C$, an element $\delta\in\E(C,A)$ is called an $\E$-extension. Zero element $0\in\E(C,A)$ is called the split $\E$-extension.
Since $\E$ is a bifunctor, for any $a\in\C(A,A')$ and $c\in\C(C',C)$, we have $\E$-extensions
$$ \E(C,a)(\del)\in\E(C,A')\ \ \text{and}\ \ \ \E(c,A)(\del)\in\E(C',A). $$
We abbreviate them to $a_\ast\del$ and $c^\ast\del$ respectively.
For any $\delta\in \E(C,A), \delta'\in \E(C',A'),$ since $\C$ and $\E$ are additive, we can define the $\E$-extension
$$\delta\oplus\delta'\in\E(C\oplus C',A\oplus A').$$

\begin{definition}{ \rm(\cite[Definition 2.3]{Na})}\label{mo}
A \emph{morphism} $(a,c):\delta\to\delta'$ of $\E$-extensions $\delta\in \E(C,A), \delta'\in \E(C',A')$ is a pair of morphisms $a\in\C(A,A')$ and $c\in\C(C,C')$ in $\C$ satisfying $ a_\ast\del=c^\ast\del'. $
\end{definition}

Let $A,C\in\C$ be any pair of objects. Sequences of morphisms in $\C$
$$\xymatrix@C=0.7cm{A\ar[r]^{x} & B \ar[r]^{y} & C}\ \ \text{and}\ \ \ \xymatrix@C=0.7cm{A\ar[r]^{x'} & B' \ar[r]^{y'} & C}$$
are said to be \emph{equivalent} if there exists an isomorphism $b\in\C(B,B')$ which makes the following diagram commutative:
$$\xymatrix@C=20pt@R=20pt{
A \ar[r]^x \ar@{=}[d] & B\ar[r]^y \ar[d]_{\simeq}^{b} & C\ar@{=}[d]&\\
A\ar[r]^{x'} & B' \ar[r]^{y'} & C. &}$$

We denote the equivalence class of $\xymatrix@C=0.7cm{A\ar[r]^{x} & B \ar[r]^{y} & C}$ by $[\xymatrix@C=0.7cm{A\ar[r]^{x} & B \ar[r]^{y} & C}]$.

\begin{definition}{ \rm(\cite[Definition 2.9]{Na})}\label{re}
For any $\E$-extension $\delta\in\E(C,A)$, one can associate an equivalence class $\s(\delta)=[\xymatrix@C=0.7cm{A\ar[r]^{x} & B \ar[r]^{y} & C}].$  This $\s$ is called a \emph{realization} of $\E$, if it satisfies the following condition:
\begin{itemize}
\item Let $\del\in\E(C,A)$ and $\del'\in\E(C',A')$ be any pair of $\E$-extensions with $$\s(\del)=[\xymatrix@C=0.7cm{A\ar[r]^{x} & B \ar[r]^{y} & C}],\ \ \ \s(\del')=[\xymatrix@C=0.7cm{A'\ar[r]^{x'} & B'\ar[r]^{y'} & C'}].$$
Then, for any morphism $(a,c)\colon\del\to\del'$, there exists $b\in\C(B,B')$ which makes the following diagram commutative:
\begin{equation*}\label{reali}
\begin{split}
\xymatrix@C=20pt@R=20pt{
A \ar[r]^x \ar[d]^a & B\ar[r]^y \ar[d]^{b} & C\ar[d]^c&\\
A'\ar[r]^{x'} & B' \ar[r]^{y'} & C'. &}
\end{split}
\end{equation*}
\end{itemize}
In this case, we say that the sequence $\xymatrix@C=0.7cm{A\ar[r]^{x} & B \ar[r]^{y} & C}$ \emph{realizes} \ $\del$, whenever it satisfies $\s(\del)=[\xymatrix@C=0.7cm{A\ar[r]^{x} & B \ar[r]^{y} & C}]$.
In the above situation, we say that the triple $(a,b,c)$ \emph{realizes} $(a,c)$.
\end{definition}

\begin{definition}{ \rm(\cite[Definition 2.10]{Na})}\label{rea}
Let $\C, \E$ be as above. A realization $\s$ of $\E$ is said to be \emph{additive} if the following conditions are satisfied:
\begin{itemize}
\item[{\rm (i)}] For any $A,C\in\C$, the split $\E$-extension $0\in \E(C,A)$ satisfies $\s(0)=0$.
\item[{\rm (ii)}] For any pair of $\E$-extensions $\del\in\E(A,C)$ and $\del'\in\E(A',C')$, we have
                  $$\s(\del\oplus \del')=\s(\del)\oplus \s(\del').$$.
\end{itemize}
\end{definition}

\begin{definition}{\rm(\cite[Definition 2.12]{Na})}\label{ext}   
We call the triple $(\C,\E,\s)$ an \emph{extriangulated category} if the following conditions are satisfied:
\begin{itemize}
\item[{\rm (ET1)}] $\E\colon\C^{op}\times\C\to\Ab$ is a biadditive functor.

\item[{\rm (ET2)}] $\s$ is an additive realization of $\E$.

\item[{\rm (ET3)}] Let $\del\in\E(C,A)$ and $\del'\in\E(C',A')$ be any pair of $\E$-extensions, realized as
$$ \s(\del)=[\xymatrix@C=0.7cm{A\ar[r]^{x} & B \ar[r]^{y} & C}],\ \ \s(\del')=[\xymatrix@C=0.7cm{A'\ar[r]^{x'} & B' \ar[r]^{y'} & C'}]. $$
For any commutative square
$$\xymatrix@C=20pt@R=20pt{
A \ar[r]^x \ar[d]^a & B\ar[r]^y \ar[d]^{b} & C&\\
A'\ar[r]^{x'} & B' \ar[r]^{y'} & C' &}$$
in $\C$, there exists a morphism $(a,c)\colon\del\to\del'$ which is realized by $(a,b,c)$.
\item[{\rm(ET3)$\rm^{op}$}] Dual of \rm (ET3).
\item[{\rm (ET4)}] Let $\delta\in\E(A,D)$ and $\del'\in\E(B,F)$ be $\E$-extensions realized as
$$\s(\del)=[\xymatrix@C=0.7cm{A\ar[r]^{f} & B \ar[r]^{f'} & D}]\ \ \text{and}\ \ \ \s(\del')=[\xymatrix@C=0.7cm{B\ar[r]^{g} & C \ar[r]^{g'} & F}]$$
respectively. Then there exist an object $E\in\C$, a commutative diagram
$$\xymatrix@C=20pt@R=20pt{A\ar[r]^{f}\ar@{=}[d]&B\ar[r]^{f'}\ar[d]^{g}&D\ar[d]^{d}\\
A\ar[r]^{h}&C\ar[d]^{g'}\ar[r]^{h'}&E\ar[d]^{e}\\
&F\ar@{=}[r]&F}$$
in $\C$, and an $\E$-extension $\del^{''}\in\E(E,A)$ realized by $\xymatrix@C=0.7cm{A\ar[r]^{h} & C \ar[r]^{h'} & E},$ which satisfy the following compatibilities:
\begin{itemize}
\item[{\rm (i)}] $\xymatrix@C=0.7cm{D\ar[r]^{d} & E \ar[r]^{e} & F}$  realizes $f'_{\ast}\del'$,
\item[{\rm (ii)}] $d^\ast\del''=\del$,
\item[{\rm (iii)}] $f_{\ast}\del''=e^{\ast}\del'$.
\end{itemize}

\item[{\rm (ET4)$\rm^{op}$}]  Dual of \rm (ET4).
\end{itemize}
\end{definition}

Many examples can be founded in \cite{HZZ20P,INY18A,Na}.
\begin{example}\label{example-extri}
(1) Exact categories, triangulated categories and extension-closed subcategories of triangulated categories are extriangulated categories.
Extension-closed subcategories of an extriangulated category are also extriangulated categories (see \cite[Remark 2.18]{Na}).

(2) If $\mathcal{C}$ is a triangulated category with suspension functor $\Sigma$ and
$\xi$ is a proper class of triangles (see \cite{B00R} for details), then $(\mathcal{C},\mathbb{E}_{\xi},\mathfrak{s}_{\xi})$ is an extriangulated category (see \cite[Remark 3.3]{HZZ20P}).

\end{example}

We recall the following notations.

\begin{itemize}
\item[(i)] A sequence $\xymatrix@C=15pt{A\ar[r]^{x} & B \ar[r]^{y} & C}$ is called a {conflation} if it realizes some $\E$-extension $\del\in\E(C,A)$.
In which case, $\xymatrix@C=15pt{A\ar[r]^{x} & B}$ is called an {inflation} and $\xymatrix@C=15pt{B \ar[r]^{y} & C}$ is called a {deflation}. We call $\xymatrix@C=15pt{A\ar[r]^{x} & B \ar[r]^{y} & C\ar@{-->}[r]^{\del}&}$ an \emph{$\E$-triangle}.

\item[(ii)] Given an $\E$-triangle $\xymatrix@C=15pt{A\ar[r]^{x} & B \ar[r]^{y} & C\ar@{-->}[r]^{\del}&},$ we call $A$ the {cocone} of $y\xymatrix@C=15pt{\colon B\ar[r]& C},$ and denote it by ${\rm cocone}(y);$ we call $C$ the {cone} of $x\colon \xymatrix@C=15pt{A\ar[r]& B},$ and denote it by ${\rm cone}(x).$
\item[(iii)] A subcategory $\mathcal{D}$ of $\C$ is called {extension-closed} if for any $\E$-triangle $\xymatrix@C=15pt{A\ar[r]^{x} & B \ar[r]^{y} & C\ar@{-->}[r]^{\del}&} $ with $A, C\in \mathcal{D}$, we have $B\in \mathcal{D}$.
\end{itemize}

Throughout this paper,
for an extriangulated category $\mathcal{C}$, we assume the following condition, which is analogous
to the weak idempotent completeness (see \cite[Proposition 7.6]{Bu}).

\begin{condition}\label{WIC} {\rm(WIC) (see \cite[Condition 5.8]{Na})} Let $f:X\rightarrow Y$ and  $g:Y\rightarrow Z$ be any composable pair of morphisms in $\mathcal{C}$.
\begin{itemize}
\item[(1)]  If $gf$ is an inflation, then $f$ is an inflation.
\item[(2)] If $gf$ is a deflation, then $g$ is a deflation.
    \end{itemize}
\end{condition}

\begin{definition}{\rm(\cite[Definitions 3.23 and 3.25]{Na})}
Let $\mathcal{C}$ be an extriangulated category.
\begin{itemize}
\item[(1)] An object $P$ in $\mathcal{C}$ is called {\em projective} if for any $\mathbb{E}$-triangle $A\stackrel{x}{\longrightarrow}B\stackrel{y}{\longrightarrow}C\stackrel{}\dashrightarrow$ and any morphism $c$ in $\mathcal{C}(P,C)$, there exists $b$ in $\mathcal{C}(P,B)$ such that $yb=c$.
We denote the full subcategory of projective objects in $\mathcal{C}$ by $\mathcal{P}(\mathcal{C})$.
Dually, the {\em injective} objects are defined, and the full subcategory of injective objects in $\mathcal{C}$ is denoted by $\mathcal{I}(\mathcal{C})$.
\item[(2)] We say that $\mathcal{C}$ {\em has enough projectives} if for any object $M\in\mathcal{C}$, there exists an $\mathbb{E}$-triangle $A\stackrel{}{\longrightarrow}P\stackrel{}{\longrightarrow}M\stackrel{}\dashrightarrow$ satisfying $P\in\mathcal{P}(\mathcal{C})$. Dually, we define that $\mathcal{C}$ {\em has enough injectives}.
\end{itemize}
\end{definition}

\begin{remark}\label{proj-coincide}
Let $\mathcal{C}$ be a triangulated category with a proper class $\xi$.
The projective objects in the extriangulated category
$(\mathcal{C},\mathbb{E}_{\xi},\mathfrak{s}_{\xi})$ coincides with the $\xi$-projective objects of $\mathcal{C}$ in sense of Beligiannis \cite{B00R}.
In fact, it is clear that the $\xi$-projective objects of $\mathcal{C}$ are projective objects of $\mathcal{C}$.
On the other hand,
let $P\in \mathcal{P(C)}$ and $\xymatrix{A\ar[r]&B\ar[r]&C\ar@{-->}[r]&}$ be an $\mathbb{E}$-triangle.
Then $\xymatrix{\Sigma^{-1}A\ar[r]&\Sigma^{-1}B\ar[r]&\Sigma^{-1}C\ar@{-->}[r]&}$
is also an $\mathbb{E}$-triangle.
Consider the following exact sequence
$$\xymatrix{\cdots\ar[r]&\mathcal{C}(P,\Sigma^{-1}B)\ar[r]&\mathcal{C}(P,\Sigma^{-1}C)\ar[r]&
\mathcal{C}(P,A)\ar[r]&\mathcal{C}(P,B)\ar[r]&\cdots}.$$
One can see that $\xymatrix{\mathcal{C}(P,\Sigma^{-1}B)\ar[r]&\mathcal{C}(P,\Sigma^{-1}C)}$ is epic, so $\xymatrix{
\mathcal{C}(P,A)\ar[r]&\mathcal{C}(P,B)}$ is monic.
Notice that $\xymatrix{\mathcal{C}(P,B)\ar[r]&\mathcal{C}(P,C)}$ is also epic,
then $$\xymatrix{0\ar[r]&
\mathcal{C}(P,A)\ar[r]&\mathcal{C}(P,B)\ar[r]&\mathcal{C}(P,C)\ar[r]&0}$$
is exact, and hence $P$ is a $\xi$-projective object in $\mathcal{C}$.
Thus $\mathcal{P(\mathcal{C})}=\mathcal{P}(\xi)$ (where $\mathcal{P(\xi)}$ denotes the subcategory of $\mathcal{C}$ consisting of all $\xi$-projective objects in $\mathcal{C}$ (see \cite[Section 4]{B00R})).
\end{remark}

Throughout this paper, we will always assume that all extriangulated categories admit enough projective objects and injective objects.

Let $\mathcal{C}$ be an extriangulated category.
 Recall from \cite{ZZ20T} that an {\em $\mathbb{E}$-triangle sequence} is defined as a sequence
 $$\cdots {\longrightarrow}X_{n+2}\stackrel{d_{n+2}}{\longrightarrow}X_{n+1}\stackrel{d_{n+1}}{\longrightarrow} X_{n}\stackrel{d_{n}}{\longrightarrow}X_{n-1}{\longrightarrow}\cdots $$
 in $\mathcal{C}$ such that for any $n$,
 there exist $\mathbb{E}$-triangles $K_{n+1}\stackrel{g_{n}}{\longrightarrow}X_{n}\stackrel{f_{n}}{\longrightarrow}K_{n}\stackrel{}\dashrightarrow$ and the differential $d_n=g_{n-1}f_n$.




\begin{definition}\label{right}{\rm(\cite[Definition 2.8]{WWZ20R})} Let $\mathcal{C}$ be an extriangulated category.
A morphism $f$ in $\mathcal{C}$ is called {\em compatible} provided that the following condition holds:
\begin{center}
  $f$ is both an inflation and a deflation implies that $f$ is an isomorphism.
\end{center}
That is, the class of compatible morphisms is the  class
$$\{f~|~f~\text{is~not~an~inflation},~\text{or}~f~\text{is~not~a~deflation},~\text{or}~f~\text{is~an~isomorphism}\}.$$
\end{definition}

It is clear that all morphisms are compatible in an exact category. While, the compatible morphisms in a triangulated category $\mathcal{C}$ are just the isomorphisms in $\mathcal{C}$.

\begin{definition}\label{right} {\rm(\cite[Definition 2.9]{WWZ20R})}
Let $\mathcal{C}$ be an extriangulated category.
A sequence $A\stackrel{f}{\longrightarrow}B\stackrel{g}{\longrightarrow}C$ in $\mathcal{C}$ is said to be {\em right exact} if
there exists an $\mathbb{E}$-triangle $K\stackrel{h_{2}}{\longrightarrow}B\stackrel{g}{\longrightarrow}C\stackrel{}\dashrightarrow$ and a deflation $h_{1}:A\rightarrow K$ which is compatible, such that $f=h_2h_1$. Dually one can also define the {\em left exact} sequences.

A $4$-term $\mathbb{E}$-triangle sequence $\xymatrix@C=15pt{A\ar[r]^{f}&B\ar[r]^{g}&C\ar[r]^{h}&D}$ in $\mathcal{C}$ is called {\em right exact} (resp. {\em left exact}) if there exist $\mathbb{E}$-triangles $A\stackrel{f}{\longrightarrow}B\stackrel{g_1}{\longrightarrow}K\stackrel{}\dashrightarrow$
and $K\stackrel{g_{2}}{\longrightarrow}C\stackrel{h}{\longrightarrow}D\stackrel{}\dashrightarrow$ such that $g=g_2g_1$ and $g_1$ (resp. $g_2$) is compatible.
\end{definition}

Let us recall the notion of right (left) exact functors in extriangulated categories.

\begin{definition}\label{right exact}{\rm
(\cite[Definition 2.12]{WWZ20R})} Let $(\mathcal{A},\mathbb{E}_{\mathcal{A}},\mathfrak{s}_{\mathcal{A}})$ and $(\mathcal{B},\mathbb{E}_{\mathcal{B}},\mathfrak{s}_{\mathcal{B}})$ be extriangulated categories. An additive covariant functor $F:\mathcal{A}\rightarrow \mathcal{B}$ is called a {\em right exact functor} if it satisfies the following conditions:
\begin{itemize}
   \item [(1)] If $f$ is a compatible morphism in $\A$, then $Ff$ is compatible in $\B$.
   \item [(2)] If $A\stackrel{a}{\longrightarrow}B\stackrel{b}{\longrightarrow}C$ is right exact in $\mathcal{A}$, then $FA\stackrel{Fa}{\longrightarrow}FB\stackrel{Fb}{\longrightarrow}FC$ is right exact in $\mathcal{B}$. (In particular,  for any $\mathbb{E}_{\mathcal{A}}$-triangle $A\stackrel{f}{\longrightarrow}B\stackrel{g}{\longrightarrow}C\stackrel{\delta}\dashrightarrow$, there exists an $\mathbb{E}_{\mathcal{B}}$-triangle $A'\stackrel{x}{\longrightarrow}FB\stackrel{Fg}{\longrightarrow}FC\stackrel{}\dashrightarrow$ such that $Ff=xy$ and $y: FA\rightarrow A'$ is a deflation and compatible. Moreover, $A'$ is uniquely determined up to isomorphism.)
  \item [(3)] There exists a natural transformation $$\eta=\{\eta_{(C,A)}:\mathbb{E}_{\mathcal{A}}(C,A)\longrightarrow\mathbb{E}_{\mathcal{B}}(F^{op}C,A')\}_{(C,A)\in{\A}^{\rm op}\times\A}$$ such that $\mathfrak{s}_{\mathcal{B}}(\eta_{(C,A)}(\delta))=[A'\stackrel{x}{\longrightarrow}FB\stackrel{Fg}{\longrightarrow}FC]$.
 \end{itemize}

Dually, we define the {\em left exact functor} between two extriangulated categories.
\end{definition}

The extriangulated functor between two extriangulated categories has been defined in \cite{Ben} (also see \cite{WWZ20R}).
\begin{definition}\label{exact functor}{\rm(\cite{Ben} and \cite[Definition 2.13]{WWZ20R})}
Let $(\mathcal{A},\mathbb{E}_{\mathcal{A}},\mathfrak{s}_{\mathcal{A}})$ and $(\mathcal{B},\mathbb{E}_{\mathcal{B}},\mathfrak{s}_{\mathcal{B}})$ be extriangulated categories.  We say an additive covariant functor $F:\mathcal{A}\rightarrow \mathcal{B}$ is an {\em exact functor} if the following conditions hold:
\begin{itemize}
  \item [(1)] If $f$ is a compatible morphism in $\A$, then $Ff$ is compatible in $\B$.
  \item [(2)] There exists a natural transformation $$\eta=\{\eta_{(C,A)}\}_{(C,A)\in{\A}^{\rm op}\times\A}:\mathbb{E}_{\mathcal{A}}(-,-)\Rightarrow\mathbb{E}_{\mathcal{B}}(F^{\rm op}-,F-).$$
  \item [(3)] If $\mathfrak{s}_{\mathcal{A}}(\delta)=[A\stackrel{x}{\longrightarrow}B\stackrel{y}{\longrightarrow}C]$, then $\mathfrak{s}_{\mathcal{B}}(\eta_{(C,A)}(\delta))=[F(A)\stackrel{F(x)}{\longrightarrow}F(B)\stackrel{F(y)}{\longrightarrow}F(C)]$.

  \end{itemize}

\end{definition}

\begin{remark}{\rm(\cite[Proposition 2.14 and Remark 2.15]{WWZ20R})} Let $(\mathcal{A},\mathbb{E}_{\mathcal{A}},\mathfrak{s}_{\mathcal{A}})$ and $(\mathcal{B},\mathbb{E}_{\mathcal{B}},\mathfrak{s}_{\mathcal{B}})$ be extriangulated categories. An additive covariant functor $F: \mathcal{A}\rightarrow\mathcal{B}$ is exact if and only if $F$ is both left exact and right exact.
In particular,
\begin{itemize}
\item[(1)] If the categories $\mathcal{A}$ and $\mathcal{B}$ are abelian, the exact (resp. left exact, right exact) functors agree with
the usual definitions.
\item[(2)] If the categories $\mathcal{A}$ and $\mathcal{B}$ are triangulated, we know that $F$ is a left exact functor if and only if $F$ is a triangle
functor if and only if $F$ is a right exact functor.
\end{itemize}
\end{remark}

Now we recall the concept of recollements of extriangulated categories \cite{WWZ20R}, which gives a simultaneous generalization of recollements of triangulated categories and abelian categories (see \cite{BBD, Fr}).

\begin{definition}\label{def-rec}{\rm(\cite[Definition 3.1]{WWZ20R}) }
Let $\mathcal{A}$, $\mathcal{B}$ and $\mathcal{C}$ be three extriangulated categories. A \emph{recollement} of $\mathcal{B}$ relative to
$\mathcal{A}$ and $\mathcal{C}$, denoted by ($\mathcal{A}$, $\mathcal{B}$, $\mathcal{C}$), is a diagram
\begin{equation*}\label{recolle}
  \xymatrix{\mathcal{A}\ar[rr]|{i_{*}}&&\ar@/_1pc/[ll]|{i^{*}}\ar@/^1pc/[ll]|{i^{!}}\mathcal{B}
\ar[rr]|{j^{\ast}}&&\ar@/_1pc/[ll]|{j_{!}}\ar@/^1pc/[ll]|{j_{\ast}}\mathcal{C}}
\end{equation*}
given by two exact functors $i_{*},j^{\ast}$, two right exact functors $i^{\ast}$, $j_!$ and two left exact functors $i^{!}$, $j_\ast$, which satisfies the following conditions:
\begin{itemize}
  \item [(R1)] $(i^{*}, i_{\ast}, i^{!})$ and $(j_!, j^\ast, j_\ast)$ are adjoint triples.
  \item [(R2)] $\Im i_{\ast}=\Ker j^{\ast}$.
  \item [(R3)] $i_\ast$, $j_!$ and $j_\ast$ are fully faithful.
  \item [(R4)] For each $B\in\mathcal{B}$, there exists a left exact $\mathbb{E}$-triangle sequence
$$
  \xymatrix{i_\ast i^!( B)\ar[r]^-{\theta_B}&B\ar[r]^-{\vartheta_B}&j_\ast j^\ast (B)\ar[r]&i_\ast (A)}
$$
 in $\mathcal{B}$ with $A\in \mathcal{A}$, where $\theta_B$ and  $\vartheta_B$ are given by the adjunction morphisms.
  \item [(R5)] For each $B\in\mathcal{B}$, there exists a right exact $\mathbb{E}$-triangle sequence
$$
  \xymatrix{i_\ast\ar[r]( A') &j_! j^\ast (B)\ar[r]^-{\upsilon_B}&B\ar[r]^-{\nu_B}&i_\ast i^\ast (B)&}
$$
in $\mathcal{B}$ with $A'\in \mathcal{A}$, where $\upsilon_B$ and $\nu_B$ are given by the adjunction morphisms.
\end{itemize}
\end{definition}

\begin{example}\label{exam-rec-abelian}
Let $(\mathcal{A},\mathcal{B},\mathcal{C})$ be a recollement of abelian categories. Assume that $\widetilde{\mathcal{X}}$ is an extension-closed subcategory of $\mathcal{C}$, then
$\mathcal{X}:=\{B\in \mathcal{B}\mid j^{*}(B)\in \widetilde{\mathcal{X}}\}$ is an extension-closed subcategory of $\mathcal{B}$.
It is easy to check that
\begin{equation*}\label{lliz}
  \xymatrix{\mathcal{A}\ar[rr]|{i_{*}}&&\ar@/_1pc/[ll]|{i^{*}}\ar@/^1pc/[ll]|{i^{!}}\mathcal{X}
\ar[rr]|{j^{\ast}}&&\ar@/_1pc/[ll]|{j_{!}}\ar@/^1pc/[ll]|{j_{\ast}} \widetilde{\mathcal{X}}}
\end{equation*}
is a recollement of extriangulated categories.
\end{example}


We collect some properties of recollements of extriangulated categories (see \cite{WWZ20R}).

\begin{lemma}\label{lem-rec} Let ($\mathcal{A}$, $\mathcal{B}$, $\mathcal{C}$) be a recollement of extriangulated categories.

$(1)$ All the natural transformations
$$i^{\ast}i_{\ast}\Rightarrow\Id_{\A},~\Id_{\A}\Rightarrow i^{!}i_{\ast},~\Id_{\C}\Rightarrow j^{\ast}j_{!},~j^{\ast}j_{\ast}\Rightarrow\Id_{\C}$$
are natural isomorphisms.
Moreover, $i^{!}$, $i^{*}$ and $j^{*}$ are dense.

$(2)$ $i^{\ast}j_!=0$ and $i^{!}j_\ast=0$.

$(3)$ $i^{\ast}$ preserves projective objects and $i^{!}$ preserves injective objects.

$(3')$ $j_{!}$ preserves projective objects and $j_{\ast}$ preserves injective objects.

$(4)$ If $i^{!}$ (resp. $j_{\ast}$) is  exact, then $i_{\ast}$ (resp. $j^{\ast}$) preserves projective objects.

$(4')$ If $i^{\ast}$ (resp. $j_{!}$) is  exact, then $i_{\ast}$ (resp. $j^{\ast}$) preserves injective objects.






$(5)$ If $i^{!}$ is exact, then $j_{\ast}$ is exact.

$(5')$  If $i^{\ast}$ is exact, then $j_{!}$ is  exact.

$(6)$ If $i^{!}$ is exact, for each $B\in\mathcal{B}$, there is an $\mathbb{E}$-triangle
  \begin{equation*}\label{third}
  \xymatrix{i_\ast i^! (B)\ar[r]^-{\theta_B}&B\ar[r]^-{\vartheta_B}&j_\ast j^\ast (B)\ar@{-->}[r]&}
   \end{equation*}
 in $\mathcal{B}$ where $\theta_B$ and  $\vartheta_B$ are given by the adjunction morphisms.

$(6')$ If $i^{\ast}$ is exact, for each $B\in\mathcal{B}$, there is an $\mathbb{E}$-triangle
  \begin{equation*}\label{four}
  \xymatrix{ j_! j^\ast (B)\ar[r]^-{\upsilon_B}&B\ar[r]^-{\nu_B}&i_\ast i^\ast (B) \ar@{-->}[r]&}
   \end{equation*}
in $\mathcal{B}$ where $\upsilon_B$ and $\nu_B$ are given by the adjunction morphisms.
\end{lemma}



\section{Global dimensions  and recollements}

Now we introduce the notion of global dimension of extriangulated categories.

\begin{definition}
Let $\mathcal{C}$ be an extriangulated category and $C\in \mathcal{C}$. The projective dimension of $C$, denoted by $\pd_{\mathcal{C}}C$, is defined as
\begin{align*}\pd _{\mathcal{C}}C:=&\inf \{n \geq 0\mid \text{there exists an }\mathbb{E}\text{-triangle sequence }\\&P_{n}\stackrel{d_{n}}{\longrightarrow}P_{n-1}\stackrel{d_{n-1}}{\longrightarrow}\cdots \stackrel{d_{2}}{\longrightarrow} P_{1}\stackrel{d_{1}}{\longrightarrow}P_{0}\stackrel{d_{0}}{\longrightarrow}C\text{ in }\mathcal{C}\text{ with }P_{i}\in \mathcal{P}(C) \text{ for }0\leq i\leq n\}.
\end{align*}
The global dimension of $\mathcal{C}$, denoted by $\gl \mathcal{C}$, is defined as
$$\gl \mathcal{C}:=\sup\{\pd_{\mathcal{C}}C\mid C\in \mathcal{C}\}.$$
\end{definition}



 \begin{example}
(1) If $\mathcal{C}$ is an exact category, then these agree with
the usual definitions.

(2) {\rm
(see \cite[Example 4.5]{B00R})}
 If $\mathcal{C}$ is a triangulated category with a proper class $\xi$, and $C$ is an object in $\mathcal{C}$,
then $\xi\text{-}\pd C$ (the $\xi$-projective dimension of $C$) and $\xi\text{-}\gl \mathcal{C}$ (the $\xi$-global dimension of $\mathcal{C}$), defined in \cite[Section 4]{B00R}, coincides with $\pd_{\mathcal{C}}C$ and $\gl\mathcal{C}$ respectively by Remark \ref{proj-coincide}.
In particular, if $\mathcal{C}$ is a triangulated category, then $\mathcal{P(C)}$ consists of zero objects, and $\gl \mathcal{C}=\infty$. Moreover, it always has enough projectives.
Assume that $\mathcal{C}$ is a triangulated category with a proper class $\xi$ consisting of all split triangles in $\mathcal{C}$, then $\gl\mathcal{C}=\xi\text{-}\gl \mathcal{C}=0$.

(3)
Let $\Lambda$ be a finite dimensional algebra given by the quiver
\[\qquad 1\longleftarrow2\longleftarrow3. \]
and let $\operatorname{D^b}(\mod \Lambda)$ be the bounded derived category.
We know that its Auslander--Reiten quiver is as follows (see \cite[5.6]{Ha1})
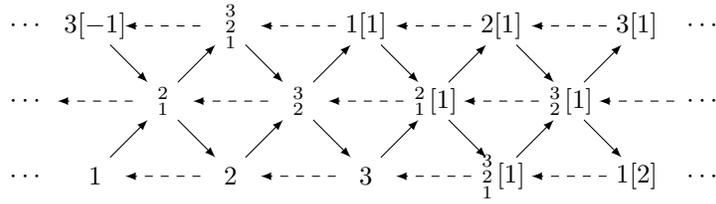
\[\begin{tikzpicture}[scale=0.5, fl/.style={->,>=latex}]
\foreach \x in {-1,0,1,2} {
   \draw[fl] (3.6*\x,1.5) -- (3.6*\x+1,0.5) ;
   \draw[fl] (3.6*\x+1.8,-0.5) -- (3.6*\x+2.8,-1.5) ;
   \draw[fl] (3.6*\x,-1.5) -- (3.6*\x+1,-0.5) ;
   \draw[fl] (3.6*\x+1.8,0.5) -- (3.6*\x+2.8,1.5) ;
   \draw[fl, dashed] (3.6*\x+2.4,2) -- (3.6*\x+0.4,2) ;
   \draw[fl, dashed] (3.6*\x+0.6,0) -- (3.6*\x-1.4,0) ;
   \draw[fl, dashed] (3.6*\x+2.4,-2) -- (3.6*\x+0.4,-2) ;
};
   \draw[fl, dashed] (3.6*3+0.6,0) -- (3.6*3-1.4,0) ;
\draw (-5.8,2) node {$\cdots$} ;
\draw (-4,2) node {$3[-1]$} ;
\draw (-0.4,2) node {$\bsm3\\2\\1\esm$} ;
\draw (3.2,2) node {$1[1]$} ;
\draw (6.8,2) node {$2[1]$} ;
\draw (10.4,2) node {$3[1]$} ;
\draw (12.2,2) node {$\cdots$} ;
\draw (-5.8,0) node {$\cdots$} ;
\draw (-2.2,0) node {$\bsm2\\1\esm$} ;
\draw (1.4,0) node {$\bsm3\\2\esm$} ;
\draw (5,0) node {$\bsm2\\1\esm[1]$} ;
\draw (8.6,0) node {$\bsm3\\2\esm[1]$} ;
\draw (12.2,0) node {$\cdots$} ;
\draw (-5.8,-2) node {$\cdots$} ;
\draw (-4,-2) node {$1$} ;
\draw (-0.4,-2) node {$2$} ;
\draw (3.2,-2) node {$3$} ;
\draw (6.8,-2) node {$\bsm3\\2\\1\esm[1]$} ;
\draw (10.4,-2) node {$1[2]$} ;
\draw (12.2,-2) node {$\cdots$} ;
\end{tikzpicture}
\]

\begin{itemize}

\item[(a)] Consider the triangulated category $\mathcal{T}:=D^{b}(\mod \Lambda)/\tau^{-1}[1]$.
The Auslander-Reiten quiver of $\mathcal{T}$ is
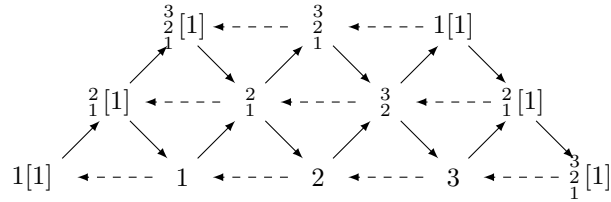
\[\begin{tikzpicture}[scale=0.5, fl/.style={->,>=latex}]

 \foreach \x in {-2,-1,0} { \draw[fl] (3.6*\x+1.8,0.5) -- (3.6*\x+2.8,1.5) ;}

 \foreach \x in {-2,-1,0,1} {\draw[fl] (3.6*\x+1.8,-0.5) -- (3.6*\x+2.8,-1.5) ; }

\foreach \x in {-1,0,1} {\draw[fl] (3.6*\x,1.5) -- (3.6*\x+1,0.5) ;}

\foreach \x in {-1,0} {\draw[fl, dashed] (3.6*\x+2.4,2) -- (3.6*\x+0.4,2) ;   }

\foreach \x in {-1,0,1} {\draw[fl, dashed] (3.6*\x+0.6,0) -- (3.6*\x-1.4,0) ;   }

\foreach \x in {-2,-1,0,1} {

   \draw[fl] (3.6*\x,-1.5) -- (3.6*\x+1,-0.5) ;

   \draw[fl, dashed] (3.6*\x+2.4,-2) -- (3.6*\x+0.4,-2) ;
};

\draw (-4,2) node {$\bsm3\\2\\1\esm[1]$} ;

\draw (-6,0) node {$\bsm2\\1\esm[1]$} ;
\draw (-8,-2) node {$1[1]$} ;
\draw (-0.4,2) node {$\bsm3\\2\\1\esm$} ;
\draw (3.2,2) node {$1[1]$} ;
%
\draw (-2.2,0) node {$\bsm2\\1\esm$} ;
\draw (1.4,0) node {$\bsm3\\2\esm$} ;
\draw (5,0) node {$\bsm2\\1\esm[1]$} ;
%
\draw (-4,-2) node {$1$} ;
\draw (-0.4,-2) node {$2$} ;
\draw (3.2,-2) node {$3$} ;
\draw (6.8,-2) node {$\bsm3\\2\\1\esm[1]$} ;
\end{tikzpicture}
\]
\begin{itemize}

\item[(a.1)] Take an extension-closed subcategory $\mathcal{C}=\add(1\oplus\bsm2\\1\esm\oplus 2)$ of $\mathcal{T}$, and thus $\mathcal{C}$ is an extriangulated category.
Clearly, $\mathcal{P(C)}=\add (\bsm2\\1\esm)$, $\mathcal{C}$ has enough projective objects, and
$\gl\mathcal{C}=\infty$.

\item[(a.2)]
 Take a rigid subcategory $\mathcal{R}=\add (\bsm2\\1\esm)$ of $\mathcal{T}$.
Then $\mathcal{C}:=\{X\in \mathcal{T}\mid\mathcal{T}(\mathcal{R},X[1])=0\}=\add (1\oplus \bsm2\\1\esm\oplus\bsm3\\2\\1\esm\oplus 2\oplus \bsm3\\2\\1\esm[1] )$ is an extriangulated category by \cite[Remark 2.18]{Na} (also see Example \ref{example-extri}).
By \cite[Example 3.26]{Na}, we know that $\mathcal{P(C)}=\mathcal{R}$, $\mathcal{C}$ has enough projective objects,
and we have $\gl\mathcal{C}=\infty$.
\end{itemize}

The following extriangulated categories can be founded in \cite{INY18A}.

\item[(b)] $(\mathcal{C},\mathbb{E},\mathfrak{s})$ is an extriangulated category with the following
Auslander-Reiten quiver
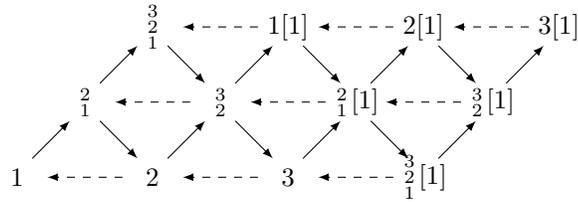
\[
\begin{tikzpicture}[scale=0.5, fl/.style={->,>=latex}]
\foreach \x in {-1,0,1,2} {
   \draw[fl] (3.6*\x+1.8,0.5) -- (3.6*\x+2.8,1.5) ;
   \draw[fl] (3.6*\x,-1.5) -- (3.6*\x+1,-0.5) ;
};
\foreach \x in {0,1,2} {
   \draw[fl] (3.6*\x,1.5) -- (3.6*\x+1,0.5) ;
   \draw[fl] (3.6*\x-1.8,-0.5) -- (3.6*\x-0.8,-1.5) ;
   \draw[fl, dashed] (3.6*\x+2.4,2) -- (3.6*\x+0.4,2) ;
   \draw[fl, dashed] (3.6*\x+0.6,0) -- (3.6*\x-1.4,0) ;
   \draw[fl, dashed] (3.6*\x-1.2,-2) -- (3.6*\x-3.2,-2) ;
};
\draw (-0.4,2) node {$\bsm3\\2\\1\esm$} ;
\draw (3.2,2) node {$1[1]$} ;
\draw (6.8,2) node {$2[1]$} ;
\draw (10.4,2) node {$3[1]$} ;
\draw (-2.2,0) node {$\bsm2\\1\esm$} ;
\draw (1.4,0) node {$\bsm3\\2\esm$} ;
\draw (5,0) node {$\bsm2\\1\esm[1]$} ;
\draw (8.6,0) node {$\bsm3\\2\esm[1]$} ;
\draw (-4,-2) node {$1$} ;
\draw (-0.4,-2) node {$2$} ;
\draw (3.2,-2) node {$3$} ;
\draw (6.8,-2) node {$\bsm3\\2\\1\esm[1]$} ;
%
\end{tikzpicture}\]
\begin{itemize}
\item[(b.1)] One can see that $\mathcal{P(C)}=\add(1\oplus\bsm2\\1\esm\oplus\bsm3\\2\\1\esm)$, $\mathcal{C}$ has enough projective objects, and $\gl\mathcal{C}=2$.
\item[(b.2)]
Take an extension-closed subcategory $\mathcal{C'}=\add(\bsm2\\1\esm\oplus \bsm3\\2\\1\esm\oplus 2\oplus\bsm3\\2\esm\oplus 1[1])$ of $\mathcal{C}$.
One can see that $\mathcal{P(C')}=\add(2\oplus\bsm2\\1\esm\oplus\bsm3\\2\\1\esm)$, $\mathcal{C'}$ has enough projective objects, and $\gl\mathcal{C'}=1$.
\end{itemize}

\item[(c)] $(\mathcal{C},\mathbb{F},\mathfrak{t})$ is an extriangulated category with the following
Auslander-Reiten quiver
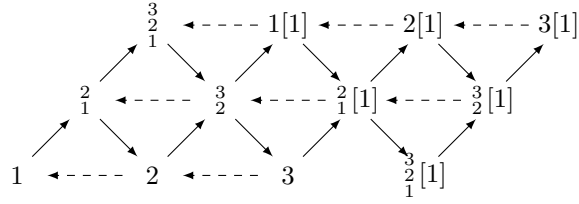
\[
\begin{tikzpicture}[scale=0.5, fl/.style={->,>=latex}]
\foreach \x in {-1,0,1,2} {
   \draw[fl] (3.6*\x+1.8,0.5) -- (3.6*\x+2.8,1.5) ;
   \draw[fl] (3.6*\x,-1.5) -- (3.6*\x+1,-0.5) ;
};
\foreach \x in {0,1,2} {
   \draw[fl] (3.6*\x,1.5) -- (3.6*\x+1,0.5) ;
   \draw[fl] (3.6*\x-1.8,-0.5) -- (3.6*\x-0.8,-1.5) ;
   \draw[fl, dashed] (3.6*\x+2.4,2) -- (3.6*\x+0.4,2) ;
   \draw[fl, dashed] (3.6*\x+0.6,0) -- (3.6*\x-1.4,0) ;
};
\foreach \x in {0,1} {
   \draw[fl, dashed] (3.6*\x-1.2,-2) -- (3.6*\x-3.2,-2) ;
};
\draw (-0.4,2) node {$\bsm3\\2\\1\esm$} ;
\draw (3.2,2) node {$1[1]$} ;
\draw (6.8,2) node {$2[1]$} ;
\draw (10.4,2) node {$3[1]$} ;
\draw (-2.2,0) node {$\bsm2\\1\esm$} ;
\draw (1.4,0) node {$\bsm3\\2\esm$} ;
\draw (5,0) node {$\bsm2\\1\esm[1]$} ;
\draw (8.6,0) node {$\bsm3\\2\esm[1]$} ;
\draw (-4,-2) node {$1$} ;
\draw (-0.4,-2) node {$2$} ;
\draw (3.2,-2) node {$3$} ;
\draw (6.8,-2) node {$\bsm3\\2\\1\esm[1]$} ;
%
\end{tikzpicture}
\]
\begin{itemize}
\item[(c.1)] One can see that $\mathcal{P(C)}=\add(1\oplus\bsm2\\1\esm\oplus\bsm3\\2\\1\esm\oplus \bsm3\\2\\1\esm[1])$, $\mathcal{C}$ has enough projective objects, and $\gl\mathcal{C}=2$.
\item[(c.2)]
Let
$\mathcal{D}=\add(\bsm3\\2\\1\esm [1])$.
Since $\bsm3\\2\\1\esm [1]$ is both-projective and injective in $\mathcal{C}$, we obtain the quotient extriangulated category $(\mathcal{C}/\mathcal{D},\mathbb{F}/\mathcal{D},\mathfrak{t}/\mathcal{D})$ from the extriangulated category $(\mathcal{C},\mathbb{F},\mathfrak{t})$ (see \cite[Proposition 3.30]{Na}).
Its Auslander-Reiten quiver is as follows

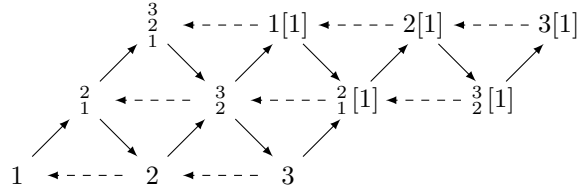
\[
\begin{tikzpicture}[scale=0.5, fl/.style={->,>=latex}]
\foreach \x in {-1,0,1,2} {
   \draw[fl] (3.6*\x+1.8,0.5) -- (3.6*\x+2.8,1.5) ;
};
\foreach \x in {-1,0,1} {
   \draw[fl] (3.6*\x,-1.5) -- (3.6*\x+1,-0.5) ;
};
\foreach \x in {0,1,2} {
   \draw[fl] (3.6*\x,1.5) -- (3.6*\x+1,0.5) ;
   \draw[fl, dashed] (3.6*\x+2.4,2) -- (3.6*\x+0.4,2) ;
   \draw[fl, dashed] (3.6*\x+0.6,0) -- (3.6*\x-1.4,0) ;
};
\foreach \x in {0,1} {
   \draw[fl] (3.6*\x-1.8,-0.5) -- (3.6*\x-0.8,-1.5) ;
   \draw[fl, dashed] (3.6*\x-1.2,-2) -- (3.6*\x-3.2,-2) ;
};
\draw (-0.4,2) node {$\bsm3\\2\\1\esm$} ;
\draw (3.2,2) node {$1[1]$} ;
\draw (6.8,2) node {$2[1]$} ;
\draw (10.4,2) node {$3[1]$} ;
\draw (-2.2,0) node {$\bsm2\\1\esm$} ;
\draw (1.4,0) node {$\bsm3\\2\esm$} ;
\draw (5,0) node {$\bsm2\\1\esm[1]$} ;
\draw (8.6,0) node {$\bsm3\\2\esm[1]$} ;
\draw (-4,-2) node {$1$} ;
\draw (-0.4,-2) node {$2$} ;
\draw (3.2,-2) node {$3$} ;
%
\end{tikzpicture}
\]
One can see that $\mathcal{P(C/\mathcal{D})}=\add(1\oplus\bsm2\\1\esm\oplus\bsm3\\2\\1\esm)$, $C/\mathcal{D}$ has enough projective objects, and $\gl\mathcal{C/\mathcal{D}}=2$.
\end{itemize}

\end{itemize}
 \end{example}

The following result is useful in the sequel.
\begin{lemma}\label{lem-3term}
Let $\mathcal{C}$ be an extriangulated category, and let
$$\xymatrix{A_{1}\ar[r]&A_{2}\ar[r]&A_{3} \ar@{-->}[r]&}$$
be an $\mathbb{E}$-triangle in $\mathcal{C}$. Then we have the following statements.
\begin{itemize}
\item[(1)] $\pd_{\mathcal{C}}A_{2}\leq\max \{\pd_{\mathcal{C}}A_{1},\pd_{\mathcal{C}}A_{3}\}$.
\item[(2)] $\pd_{\mathcal{C}}A_{3}\leq\max\{\pd_{\mathcal{C}}A_{1}+1, \pd_{\mathcal{C}}A_{2}\}$.
\item[(3)] $\pd_{\mathcal{C}}A_{1}\leq\max\{\pd_{\mathcal{C}}A_{2}, \pd_{\mathcal{C}}A_{3}-1\}$.
\end{itemize}
\end{lemma}
\begin{proof}
(1) Apply \cite[Lemmas 4.14 and 4.15]{HZZ20P}.

(2) Apply (ET4)$\rm^{op}$ (see \cite[Remark 2.22]{Na}) and (1).

(3) Apply \cite[Proposition 3.15]{Na} and (1).
\end{proof}

Let $(\mathcal{A},\mathcal{B},\mathcal{C})$
 be a recollement of extriangulated categories. Set
$$\gl_{\mathcal{A}}\mathcal{B}:=\{\pd_{\mathcal{B}}i_{*}(A)\mid A\in \mathcal{A}\}.$$
Clearly, $\gl_{\mathcal{A}}\mathcal{B}\leq \gl\mathcal{B}$.

\begin{lemma}\label{lem-j_{!}}
Let $(\mathcal{A},\mathcal{B},\mathcal{C})$ be a recollement of extriangulated categories, and let $C\in\mathcal{C}$. Then ${\pd_{\mathcal{B}}}j_{!}(C)\leq {\pd_{\mathcal{C}}}C+{\gl _{\mathcal{A}}}\mathcal{B}+1$.
\end{lemma}
\begin{proof}
If ${\gl _{\mathcal{A}}}\mathcal{B}=\infty$ or ${\pd_{\mathcal{C}}}C=\infty$, there is nothing to prove.
Assume that ${\gl _{\mathcal{A}}}\mathcal{B}=n$.
We proceed by induction on the projective dimension of $C$.
If $C$ is projective, then $j_{!}(C)\in \mathcal{P(B)}$ by Lemma \ref{lem-rec}, and so the result follows.
Now suppose $\pd_{\mathcal{C}}C=m\geq 1$.
Consider the following $\mathbb{E}$-triangle sequence
$$\xymatrix@C=15pt{P_{m}\ar[rr]&&\cdots\ar[rr]&&P_{2}\ar[rr]&&P_{1}\ar[rr]\ar[rd]&&P_{0}\ar[rr]^{g}&&C\\
&&&&&&&K_{1}\ar[ru]_{f}&&&}$$
in $\mathcal{C}$ with $P_{i}\in \mathcal{P(C)}$ for $0\leq i\leq m$.
Notice that $\pd_{\mathcal{C}}K_{1}\leq m-1$, by induction hypothesis, we have
$$\pd_{\mathcal{B}}j_{!}(K_{1})\leq \pd_{\mathcal{C}}K_{1}+n+1\leq m-1+n+1=m+n.$$
Since $j_{!}$ is right exact, there is an $\mathbb{E}$-triangle $K_{1}'\stackrel{h_{2}}{\longrightarrow}j_{!}(P_{0})\stackrel{j_{!}(g)}{\longrightarrow}j_{!}(C)\stackrel{}\dashrightarrow$
in $\mathcal{B}$ and a deflation $\xymatrix@C=15pt{h_{1}:j_{!}(K_{1})\ar[r]&K'_{1}}$ which is compatible, such that $j_{!}(f)=h_{2}h_{1}$.

Since $j^{*}j_{!}\Rightarrow {\rm Id}_{\mathcal{C}}$ by Lemma \ref{lem-rec},
$$\xymatrix@C=20pt{j^{*}j_{!}(K_{1})\ar[r]^{j^{*}j_{!}(f)}&j^{*}j_{!}(P_{0})\ar[r]^{j^{*}j_{!}(g)}&j^{*}j_{!}(C)\ar@{-->}[r]&}$$
is an $\mathbb{E}$-triangle in $\mathcal{C}$.
Since $f=j^{*}j_{!}(f)=(j^{*}(h_{2}))(j^{*}(h_{1}))$, so $j^{*}(h_{1})$ is an inflation by Condition \ref{WIC}.
Notice that $j^{*}(h_{1})$ is a deflation and compatible since $j^{*}$ is exact, so $j^{*}(h_{1})$ is an isomorphism.
So $j^{*}j_{!}(K_{1})\cong j^{*}(K'_{1})$.
Set $K''_{1}=\cocone (h_{1})$, consider the following $\mathbb{E}$-triangle
\begin{align}\label{E-triangle-1}
\xymatrix@C=20pt{K''_{1}\ar[r]&j_{!}(K_{1})\ar[r]^{h_1}&K'_{1}\ar@{-->}[r]&}
\end{align}
in $\mathcal{B}$.
Since $j^{*}$ is exact,  $j^{*}(K''_{1})=0$. By (R2), there exists an object $A'\in \mathcal{A}$ such that $K''_{1}\cong i_{*}(A')$.
Then $\pd_{\mathcal{B}}K''_{1}\leq n$ by assumption.
Apply Lemma \ref{lem-3term} to the $\mathbb{E}$-triangle (\ref{E-triangle-1}),
we have $\pd_{\mathcal{B}}K'_{1}\leq m+n$.
It follows that $\pd_{\mathcal{B}}j_{!}(C)\leq m+n+1$ from the fact that $j_{!}(P_{0})\in \mathcal{P(B)}$.
 \end{proof}

\begin{theorem}\label{main-gl}
Let $(\mathcal{A},\mathcal{B},\mathcal{C})$ be a recollement of extriangulated categories, and let $C\in\mathcal{C}$. Then we have the following statements.
\begin{itemize}
\item[(1)]
\begin{itemize}
\item[(a)] ${\gl \mathcal{B}}\leq {\gl _{\mathcal{A}}}\mathcal{B}+\gl \mathcal{C}+1$.
\item[(b)] ${\gl _{\mathcal{A}}}\mathcal{B}\leq \gl \mathcal{A}+\sup\{{\pd_{\mathcal{B}}}i_{*}(P)\mid P\in \mathcal{P}(\mathcal{A})\}$.
\item[(c)] $\gl\mathcal{B}\leq \gl \mathcal{A}+\gl\mathcal{C}+\sup\{\pd_{\mathcal{B}}i_{*}(P)\mid P\in \mathcal{P(A)}\}+1.$
\item [(d)] $\gl \mathcal{C}\leq \gl \mathcal{B}+\sup\{{\pd_{\mathcal{C}}}j^{*}(P)\mid P\in \mathcal{P}(\mathcal{B})\}$.
\end{itemize}

\item[(2)] Assume that $i^{!}$ is exact, then
\begin{itemize}
\item[(a)] $\gl\mathcal{B}\leq \gl \mathcal{A}+\gl\mathcal{C}+1.$
\item[(b)] $\gl \mathcal{C}\leq \gl\mathcal{B}.$
\end{itemize}
\end{itemize}

\end{theorem}
\begin{proof}
(1) (a)
Suppose that $\gl_{\mathcal{A}}\mathcal{B}=n< \infty$ and $\gl \mathcal{C}=m< \infty$.
Let $B\in \mathcal{B}$.
By (R5),
there exists a commutative diagram
\begin{equation*}
\xymatrix{
  &i_{\ast}(A') \ar[r]&j_{!}j^{\ast}(B)\ar[rr]^-{\upsilon_B}\ar[dr]_{h_{2}}&  &B\ar[r]^-{h}&i_{\ast}i^{\ast}(B) &\\
           &                &       &  B' \ar[ur]_{h_{1}}& }
\end{equation*}
in $\mathcal{B}$ such that $i_{\ast}A'\stackrel{}{\longrightarrow}j_{!}j^{\ast}B\stackrel{h_{2}}{\longrightarrow}B'\stackrel{}\dashrightarrow$ and $B'\stackrel{h_{1}}{\longrightarrow}B\stackrel{h}{\longrightarrow}i_{\ast}i^{\ast}B\stackrel{}\dashrightarrow$ are $\mathbb{E}$-triangles and $h_{2}$ is compatible.
Notice that
$\pd_{\mathcal{B}}i_{*}(A')\leq n$ and $\pd_{\mathcal{B}}i_{*}i^{*}(B)\leq n$.
By Lemmas \ref{lem-3term} and \ref{lem-j_{!}},
\begin{align*}
\pd_{\mathcal{B}}B&\leq \max\{\pd_{\mathcal{B}}B',\pd_{\mathcal{B}}i_{*}i^{*}(B)\}\\
&\leq \max\{\pd_{\mathcal{B}}i_{*}(A')+1,\pd_{\mathcal{B}}j_{!}j^{*}(B),\pd_{\mathcal{B}}i_{*}i^{*}(B)\}\\
&\leq \max\{n+1,\pd_{\mathcal{C}}j^{*}(B)+n+1,n\}
\end{align*}
Notice that $\pd_{\mathcal{C}}j^{*}(B)\leq m$, so $\pd_{\mathcal{B}}B\leq m+n+1$.

(b) Suppose that $\sup\{{\pd_{\mathcal{B}}}i_{*}(P)\mid P\in \mathcal{P}(\mathcal{A})\}=n<\infty$ and $\gl \mathcal{A}=m< \infty$.
Let $A\in \mathcal{A}$.
If $A$ is projective, then $\pd_{\mathcal{B}}i_{*}(A)\leq n$ and our result holds.
Now suppose that $\pd_{\mathcal{A}}A=s\leq m$.
Consider the following $\mathbb{E}$-triangle sequence
$$\xymatrix{P'_{s}\ar[r]&P'_{s-1}\ar[r]&\cdots\ar[r]&P'_{1}\ar[r]&P'_{0}\ar[r]&A}$$
in $\mathcal{A}$ with $P'_{i}\in \mathcal{P(A)}$ for $0\leq i\leq s$.
Since $i_{*}$ is exact,
$$\xymatrix{i_{*}(P'_{s})\ar[r]&i_{*}(P'_{s-1})\ar[r]&\cdots\ar[r]&i_{*}(P'_{1})\ar[r]&i_{*}(P'_{0})\ar[r]&i_{*}(A)}$$
is an $\mathbb{E}$-triangle sequence in $\mathcal{B}$.
Notice that $\pd _{\mathcal{B}}i_{*}(P'_{i})\leq n$ by assumption,
so $\pd_{\mathcal{B}}i_{*}(A)\leq s+n \leq m+n$ by Lemma \ref{lem-3term}.

(c) It follows from (a) and (b).

(d) Suppose that $\gl\mathcal{B}= m<\infty$ and $\sup\{{\pd_{\mathcal{C}}}j^{*}(P)\mid P\in \mathcal{P}(\mathcal{B})\}=n<\infty$.
For any object $C\in \mathcal{C}$, $j_{!}(C)\in \mathcal{B}$. Assume that $\pd _{\mathcal{B}}j_!(C)=s\leq m$,
and consider the following $\mathbb{E}$-triangle sequence
$$\xymatrix{P_{s}\ar[r]&P_{s-1}\ar[r]&\cdots\ar[r]&P_{1}\ar[r]&P_{0}\ar[r]&j_{!}(C)}$$
in $\mathcal{B}$ with $P_{i}\in \mathcal{P(B)}$ for $0\leq i\leq s$.

Since $j^{*}$ is exact,
$$\xymatrix{j^{*}(P_{s})\ar[r]&j^{*}(P_{s-1})\ar[r]&\cdots\ar[r]&j^{*}(P_{1})\ar[r]&j^{*}(P_{0})\ar[r]&C(\cong j^{*}j_{!}(C))}$$
is an $\mathbb{E}$-triangle sequence in $\mathcal{C}$.
Notice that $\pd _{\mathcal{C}}j^{*}(P_{i})\leq n$ by assumption,
so $\pd_{\mathcal{C}}C\leq s+n \leq m+n$ by Lemma \ref{lem-3term}.

(2)
(a) Assume that $i^{!}$ is exact, then $i_{*}$ preserves projective objects by Lemma \ref{lem-rec}.
In this case, $\gl_{\mathcal{A}}\mathcal{B}\leq\gl \mathcal{A}$.
The assertion follows from (1)(a).

(b) Since $i^{!}$ is exact,
$j^{*}$ preserves projective objects by Lemma \ref{lem-rec}.
The assertion follows from (1)(d).
\end{proof}

Immediately, we have
\begin{corollary}
Let $(\mathcal{A},\mathcal{B},\mathcal{C})$ be a recollement of extriangulated categories. Then we have the following statements.
\begin{itemize}
\item[(1)] If $\gl \mathcal{C}< \infty$, then ${\gl _{\mathcal{A}}}\mathcal{B}< \infty$ if and only if $\gl \mathcal{B}< \infty$.
\item[(2)] If $\gl\mathcal{A}<\infty$, then
${\gl _{\mathcal{A}}}\mathcal{B}<\infty$ if and only if $\sup\{{\pd_{\mathcal{B}}}i_{*}(P)\mid P\in \mathcal{P}(\mathcal{A})\}<\infty$.
\item[(3)] If $\gl \mathcal{B}<\infty$, then
$\gl \mathcal{C}<\infty$ if and only if $\sup\{{\pd_{\mathcal{B}}}j^{*}(P)\mid P\in \mathcal{P}(\mathcal{B})\}<\infty$.
\end{itemize}
\end{corollary}

\begin{corollary}{\rm(see \cite[Theorem 4.1 and Proposition 4.4]{PC14H})}
Let ($\mathcal{A}$, $\mathcal{B}$, $\mathcal{C}$) be a recollement of abelian categories.
Then
$$\gl\mathcal{B}\leq \gl \mathcal{A}+\gl\mathcal{C}+\sup\{\pd_{\mathcal{B}}i_{*}(P)\mid P\in \mathcal{P(A)}\}+1.$$
In addition, if $i^{!}$ is exact, then
$$\gl \mathcal{C}\leq \gl\mathcal{B}\leq \gl\mathcal{A}+\gl\mathcal{C}+1.$$
\end{corollary}


\section{Extension dimensions and recollements}
Let $\mathcal{C}$ be an extriangulated category,
and let $\mathcal{U}_1,\mathcal{U}_2,\cdots,\mathcal{U}_n$ be subcategories of $\mathcal{C}$. Define
\begin{align*}
\mathcal{U}_1\diamond \mathcal{U}_2:={\add}\{A\in \mathcal{C}\mid {\rm there \;exists \;an\; \mathbb{E}\text{-}triangle\ \ \ \ \ \ \ \ \ \ \ \ \ \ \ \ \ \ \ \ \ \ \ \ \ \ \ \ \ \  }
\\
 \xymatrix{U_1\ar[r]&  A \ar[r]& U_2\ar@{-->}[r]&} \ {\rm in}\ \mathcal{C}\ {\rm with}\; U_1 \in \mathcal{U}_1 \;{\rm and}\;
U_2 \in \mathcal{U}_2\}.
\end{align*}
By (ET4) and (ET4)$^{\rm op}$, the operator $\diamond$ is associative, that is,
$$(\mathcal{U}_{1}\diamond\mathcal{U}_{2})\diamond\mathcal{U}_{3}=\mathcal{U}_{1}
\diamond(\mathcal{U}_{2}\diamond\mathcal{U}_{3}).$$
Then the subcategory $\mathcal{U}_{1}\diamond  \mathcal{U}_{2}\diamond \dots \diamond\mathcal{U}_{n}$ can be inductively defined as follows
\begin{align*}
\mathcal{U}_{1}\diamond  \mathcal{U}_{2}\diamond \dots \diamond\mathcal{U}_{n}:=
\add \{A\in \mathcal{C}\mid {\rm there \;exists \;an\; \mathbb{E}\text{-}triangle \ \ \ \ \ \ \ \ \ \ \ \ \ \ \ \ \ \ \ \ \ \ \ \ \ \ \ }\\ \xymatrix{U\ar[r]&  A \ar[r]& V\ar@{-->}[r]&}
{\rm\ in}\ \mathcal{C}\ {\rm with}\; U \in \mathcal{U}_{1} \;{\rm and}\;
V \in  \mathcal{U}_{2}\diamond \dots \diamond\mathcal{U}_{n}\}.
\end{align*}
For a class of objects $\mathcal{U}$ of $\mathcal{C}$, set
$\langle\mathcal{U}\rangle_{0}=0$, $\langle\mathcal{U}\rangle_{1}=\add\mathcal{U}$,
$\langle\mathcal{U}\rangle_{n}=\langle\mathcal{U}\rangle_1\diamond \langle\mathcal{U}\rangle_{n-1}$ for any $n\geq 2$,
and $\langle\mathcal{U}\rangle_{\infty}=\mathop{\bigcup}_{n\geq 0}\langle\mathcal{U}\rangle_{n}$.
If $T$ is an object in $\mathcal{C}$, we write $\langle T\rangle_{n}:=\langle \add T \rangle_{n}$.

Now, we introduce the notion of extension dimension of extriangulated categories, which unifies the notion of dimension of triangulated categories (\cite{RR08D}) and the extension dimension of abelian categories (see \cite{BA08S, DHL, ZJL}).
\begin{definition}
Let $\mathcal{C}$ be an extriangulated category. The extension dimension of $\mathcal{C}$, denoted by ${\rm ext}.\dim \mathcal{C}$, is defined as
$${\rm ext}.\dim \mathcal{C}:=\inf\{n\geq 0\mid \text{ there exists } T\in \mathcal{C}\text{ such that }\langle T\rangle_{n+1}=\mathcal{C}\}.$$
\end{definition}

\begin{lemma}\label{lemma-oplus}
Let $\mathcal{C}$ be an extriangulated category. Then
for any $T_{1},T_{2}\in \mathcal{C}$ and $m,n\geq 1$, we have
$\langle T_{1}\rangle_{m}\diamond \langle T_{2}\rangle_{n}\subseteq \langle T_{1}\oplus T_{2}\rangle_{m+n}$.
\end{lemma}

\begin{proof}
Clearly,
$\langle T_{1}\rangle_{m}\subseteq \langle T_{1}\oplus T_{2}\rangle_{m}$ and $\langle T_{2}\rangle_{n}\subseteq \langle T_{1}\oplus T_{2}\rangle_{n}$.
Thus
$\langle T_{1}\rangle_{m}\diamond \langle T_{2}\rangle_{n}\subseteq
\langle T_{1}\oplus T_{2}\rangle_{m}\diamond \langle T_{1}\oplus T_{2}\rangle_{n}
=\langle T_{1}\oplus T_{2}\rangle_{m+n}$.
\end{proof}


We need the following fact.

\begin{lemma}\label{lem-Fexact}
Let $F:\mathcal{A}\longrightarrow \mathcal{B}$ be an exact functor of extriangulated categories.
Then $F(\langle T\rangle_{n})\subseteq\langle F(T)\rangle_{n}$ for any $T\in \mathcal{A}$ and $n\geq 1$.
\end{lemma}

\begin{proof}
We proceed by induction on $n$.
Let $X\in F(\langle T\rangle_{1})$. Then $X=F(Y)$ for some $Y\in\langle T\rangle_{1}(=\add T)$,
there exist some $Z\in\mathcal{A}$ and some integer $l\geq 1$ such that
$Y\oplus Z\cong T^{l}$.
Then
$$X\oplus F(Z)=F(Y)\oplus F(Z)\cong F(Y\oplus Z)\cong F(T^{l})\cong F(T)^{l}.$$
So $X\in\langle F(T)\rangle_{1}$ and $F(\langle T\rangle_{1})\subseteq \langle F(T)\rangle_{1}$.
The case $n=1$ is proved.

Now let $X\in F(\langle T\rangle_{n})$ with $n\geq 2$. Then $X=F(Y)$ for some $Y\in \langle T\rangle_{n}$,
there exist some object $Y'\in \mathcal{A}$ and an $\mathbb{E}$-triangle
$$\xymatrix{T_{1} \ar[r]&Y\oplus Y'\ar[r]& T_{2} \ar@{-->}[r]&}$$
in $\mathcal{A}$ with $T_{1}\in \langle T\rangle_{1}$ and $T_{2}\in \langle T\rangle_{n-1}$.
Since $F$ is exact, we get the following $\mathbb{E}$-triangle
$$\xymatrix{F(T_{1}) \ar[r]& F(Y)\oplus F(Y')\ar[r]& F(T_{2}) \ar@{-->}[r]&}$$
in $\mathcal{B}$.
By the induction hypothesis, $F(T_{1})\in F(\langle T\rangle_{1})\subseteq\langle F(T)\rangle_{1}$ and
$F(T_{2})\in F(\langle T\rangle_{n-1})\subseteq\langle F(T)\rangle_{n-1}$. It follows that
$$X=F(Y)\in
\langle F(T)\rangle_{1}\diamond\langle F(T)\rangle_{n-1}=\langle F(T)\rangle_{n}.$$
Thus $F(\langle T\rangle_{n})\subseteq\langle F(T)\rangle_{n}$.
\end{proof}

Let $F:\mathcal{A}\longrightarrow \mathcal{B}$ be a functor of additive categories. Recall from \cite[p.629]{XCC06A}
that $F$ is called {\em quasi-dense} if for any $B\in \mathcal{B}$, there exists $A\in \mathcal{A}$ such that $B$
is isomorphic to a direct summand of $F(A)$. Obviously, any dense functor is quasi-dense.

\begin{lemma}\label{lem-dense}
Let $F:\mathcal{A}\longrightarrow \mathcal{B}$ be an exact functor of extriangulated categories.
If $F$ is quasi-dense, then ${\rm ext}.\dim\mathcal{A}\geq {\rm ext}.\dim \mathcal{B}$.
\end{lemma}
\begin{proof}
Suppose ${\rm ext}.\dim \mathcal{A}=n$, that is, $\mathcal{A}=\langle T\rangle_{n+1}$ for some $T\in \mathcal{A}$.
Let $X\in\mathcal{B}$. Since $F$ is quasi-dense, we have $X\oplus X_{1}\cong F(Y)$ for some $Y\in\mathcal{A}$ and $X_{1}\in\mathcal{B}$.
It follows that $X\oplus X_{1}\in F(\mathcal{A})=  F(\langle T\rangle_{n+1})\subseteq\langle F(T)\rangle_{n+1}$ by Lemma \ref{lem-Fexact}.
So $X\in\langle F(T)\rangle_{n+1}$ and $\mathcal{B}\subseteq\langle F(T)\rangle_{n+1}$, which implies
${\rm ext}.\dim \mathcal{B}\leq n$.
Thus ${\rm ext}.\dim \mathcal{B}\leq {\rm ext}.\dim \mathcal{A}$.
\end{proof}

\begin{theorem}
Let $(\mathcal{A},\mathcal{B},\mathcal{C})$ be a recollement of extriangulated categories.
If either $i^{!}$ or $i^{*}$
is exact, then
$$\max\{{\rm ext}.\dim \mathcal{A},{\rm ext}.\dim \mathcal{C}\} \leq {\rm ext}.\dim \mathcal{B}\leq {\rm ext}.\dim \mathcal{A} +{\rm ext}.\dim \mathcal{C}+1.$$
\end{theorem}
\begin{proof}
Assume that $i^{!}$ is exact.
By Definition \ref{def-rec} and Lemma \ref{lem-rec}, we know that
$i^{!}$ and $j^{*}$ are exact and dense.
Then, by Lemma \ref{lem-dense}, we have
$\max\{{\rm ext}.\dim \mathcal{A},{\rm ext}.\dim \mathcal{C}\} \leq {\rm ext}.\dim \mathcal{B}$.

Let ${\rm ext}.\dim \mathcal{A}=n$ and ${\rm ext}.\dim \mathcal{C}=m$. Then there exist $X\in\mathcal{A}$ and $Y\in\mathcal{C}$
such that $\mathcal{A}=\langle X\rangle_{n+1}$ and $\mathcal{C}=\langle Y\rangle_{m+1}$. Let $B\in\mathcal{B}$.
Since $i^{!}$ is exact by assumption, we have an $\mathbb{E}$-triangle
$$\xymatrix@C=15pt{i_{*}i^{!}(B)\ar[r]&B\ar[r]&
j_{*}j^{*}(B) \ar@{-->}[r]&}$$
in $\mathcal{B}$. Note that $j_{*}$ is exact by Lemma \ref{lem-rec} and
$i_{*}$ is exact by Definition \ref{def-rec}.
Since $i^{!}(B)\in \mathcal{A}=\langle X\rangle_{n+1}$ and $j^{*}(B)\in\mathcal{C}=\langle Y\rangle_{m+1}$,
we have $i_{*}i^{!}(B)\in i_{*}(\langle X\rangle_{n+1})\subseteq \langle i_{*}(X)\rangle_{n+1}$ and $j_{*}j^{*}(B)\in j_{*}(\langle Y\rangle_{m+1})\subseteq \langle j_{*}(Y)\rangle_{m+1}$
by Lemma \ref{lem-Fexact}.
Thus $B\in \langle i_{*}(X)\rangle_{n+1}\diamond \langle j_{*}(Y)\rangle_{m+1} \subseteq
\langle i_{*}(X)\oplus j_{*}(Y)\rangle_{n+m+2}$ by Lemma \ref{lemma-oplus},
and therefore
${\rm ext}.\dim \mathcal{B}\leq n+m+1$.

For the case that $i^{*}$ is exact, the argument is similar.
\end{proof}

\begin{corollary}{\rm(\cite[Theorem 7.4]{PC14H})}
Let $(\mathcal{A},\mathcal{B},\mathcal{C})$ be a recollement of triangulated categories. Then
$$\max\{{\rm ext}.\dim \mathcal{A},{\rm ext}.\dim \mathcal{C}\} \leq {\rm ext}.\dim \mathcal{B}\leq {\rm ext}.\dim \mathcal{A} +{\rm ext}.\dim \mathcal{C}+1.$$
\end{corollary}

\begin{corollary}{\rm(\cite[Theorem 5.5]{ZJL})}
Let $(\mathcal{A},\mathcal{B},\mathcal{C})$ be a recollement of abelian categories. If either $i^{!}$ or $i^{*}$
is exact, then
$$\max\{{\rm ext}.\dim \mathcal{A},{\rm ext}.\dim \mathcal{C}\} \leq {\rm ext}.\dim \mathcal{B}\leq {\rm ext}.\dim \mathcal{A} +{\rm ext}.\dim \mathcal{C}+1.$$
\end{corollary}

\section{Examples}

Let $\Lambda', \Lambda''$ be artin algebras and $_{\Lambda'}M_{\Lambda''}$ an $(\Lambda',\Lambda'')$-bimodule, and let $\Lambda={\Lambda'\ {M}\choose\  0\  \ \Lambda''}$ be a triangular matrix algebra.
Then any module in $\mod \Lambda$ can be uniquely written as a triple ${X\choose Y}_{f}$ with $X\in\mod \Lambda'$, $Y\in\mod \Lambda''$
and $f\in\Hom_{\Lambda'}(M\otimes_{\Lambda''}Y,X)$ (see \cite[p.76]{AMRISSO95R} for more details).

 Let $\Lambda'$ be a finite dimensional algebra given by the quiver $\xymatrix@C=15pt{1\ar[r]&2}$ and $\Lambda''$ be a finite dimensional algebra given by the quiver $\xymatrix@C=15pt{3\ar[r]^{\alpha}&4\ar[r]^{\beta}&5}$ with the relation $\beta\alpha=0$. Define a triangular matrix algebra $\Lambda={\Lambda'\ \Lambda'\choose \ 0\ \ \Lambda''}$, where the right $\Lambda''$-module structure on $\Lambda'$ is induced by the unique algebra surjective homomorphsim $\xymatrix@C=15pt{\Lambda''\ar[r]^{\phi}&\Lambda'}$ satisfying $\phi(e_{3})=e_{1}$, $\phi(e_{4})=e_{2}$, $\phi(e_{5})=0$.  Then $\Lambda$ is
a finite dimensional algebra given by the quiver
$$
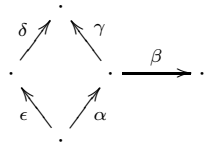
\xymatrix@C=10pt@R=15pt{&\cdot\\
\cdot\ar[ru]^{\delta}&&\ar[lu]_{\gamma}\cdot\ar[rr]^-{\beta}&&\cdot\\
&\ar[lu]^{\epsilon}\cdot\ar[ru]_{\alpha}}$$
with the relations $\gamma\alpha=\delta\epsilon$ and $\beta\alpha=0$. The Auslander-Reiten quiver of $\Lambda$ is
$$\xymatrix@C=10pt@R=12pt{{0\choose P_5}\ar[rd]&&{S_2\choose S_4}\ar[rd]&&{S_1\choose 0}\ar[rd]&&0\choose P_3\ar[rd]\\
&{S_2\choose P_4}\ar[ru]\ar[rd]&&P_1\choose S_4\ar[ru]\ar[r]\ar[rd]&P_1\choose P_3\ar[r]&S_1\choose P_3\ar[ru]\ar[rd]&&{0\choose S_3}.\\
S_2\choose 0\ar[ru]\ar[rd]&&P_1\choose P_4\ar[ru]\ar[rd]&&0\choose S_4\ar[ru]&&S_1\choose S_3\ar[ru]\\
&P_1\choose 0\ar[ru]&&0\choose P_4\ar[ru]}$$
By \cite[Example 2.12]{PC14H},
$$\xymatrix{\mod \Lambda'\ar[rr]!R|-{i_{*}}&&\ar@<-2ex>[ll]!R|-{i^{*}}
\ar@<2ex>[ll]!R|-{i^{!}}\mod \Lambda
\ar[rr]!L|-{j^{*}}&&\ar@<-2ex>[ll]!R|-{j_{!}}\ar@<2ex>[ll]!R|-{j_{*}}
\mod \Lambda''}$$
is a recollement of module categories, where
\begin{equation}\label{equation-exam}
\begin{aligned}
&i^{*}({X\choose Y}_{f})=\Coker f,\ \ \ \ \ \ \  && i_{*}(X)={X\choose 0},\ \ \ \ \ \ \  &&i^{!}({X\choose Y}_{f})=X,\\
&j_{!}(Y)={M\otimes_{\Lambda''} Y\choose Y}_{1},\ \ \ \ \ \  \ & & j^{*}({X\choose Y}_{f})=Y,\ \ \ \ \ \ \ \  &&j_{*}(Y)={0\choose Y}.
\end{aligned}
\end{equation}
\begin{itemize}
\item[(1)] Take $\widetilde{\mathcal{X}}=\add (P_3\oplus S_3)$.
Then by Example \ref{exam-rec-abelian}, $(\mod\Lambda',\mathcal{X},\widetilde{\mathcal{X}})$ is a recollement of extriangulated categories, where $\mathcal{X}=\add({S_{2}\choose 0}\oplus{P_{1}\choose 0}\oplus{S_1\choose 0}\oplus{P_1\choose P_3}\oplus{S_1\choose P_3}\oplus{S_1\choose S_3}\oplus{0\choose P_3}\oplus {0\choose S_3})$.
Notice that $\gl \mod \Lambda=1$ and $\gl \widetilde{\mathcal{X}}=0$. By Theorem \ref{main-gl}, we have $0\leq\gl\mathcal{X}\leq 2$, in fact $\gl \mathcal{X}=2$.

\item[(2)] Take $\widetilde{\mathcal{X}}=\add (P_5\oplus P_4\oplus S_4)$.
Then by Example \ref{exam-rec-abelian}, $(\mod\Lambda',\mathcal{X},\widetilde{\mathcal{X}})$ is a recollement of extriangulated categories, where $\mathcal{X}=\add({0\choose P_5}\oplus{S_{2}\choose 0}\oplus{P_{1}\choose 0}\oplus{S_2\choose P_4}\oplus{P_1\choose P_4}\oplus{S_2\choose S_4}\oplus{P_1\choose S_4}\oplus{0\choose P_4}\oplus{0\choose S_4}\oplus {S_1\choose 0})$.
Notice that $\gl \mod \Lambda=1$ and $\gl \widetilde{\mathcal{X}}=1$. By Theorem \ref{main-gl}, we have $1\leq\gl\mathcal{X}\leq 3$, in fact $\gl \mathcal{X}=1$.

\item[(3)] Take $\widetilde{\mathcal{X}}=\add (S_4\oplus P_3\oplus S_3)$.
Then by Example \ref{exam-rec-abelian}, $(\mod\Lambda',\mathcal{X},\widetilde{\mathcal{X}})$ is a recollement of extriangulated categories, where $\mathcal{X}=\add({S_{2}\choose 0}\oplus{P_{1}\choose 0}\oplus{S_2\choose S_4}\oplus{P_1\choose S_4}\oplus {S_1\choose 0}\oplus {P_1\choose P_3}\oplus{0\choose S_4}\oplus{S_1\choose P_3}\oplus{S_1\choose S_3}\oplus{0\choose P_3}\oplus {0\choose S_3})$.
Notice that $\gl \mod \Lambda=1$ and $\gl \widetilde{\mathcal{X}}=1$. By Theorem \ref{main-gl}, we have $1\leq \gl\mathcal{X}\leq 3$, in fact $\gl \mathcal{X}=2$.

\item[(4)] Take $\widetilde{\mathcal{X}}=\add (S_3)$.
Then by Example \ref{exam-rec-abelian}, $(\mod\Lambda',\mathcal{X},\widetilde{\mathcal{X}})$ is a recollement of extriangulated categories, where $\mathcal{X}=\add({S_{2}\choose 0}\oplus{P_{1}\choose 0}\oplus{S_1 \choose 0}\oplus{S_1\choose S_3}\oplus {0\choose S_3})$.
Notice that $\gl \mod \Lambda=1$ and $\gl \widetilde{\mathcal{X}}=0$. By Theorem \ref{main-gl}, we have $0\leq \gl\mathcal{X}\leq 2$, in fact $\gl \mathcal{X}=2$.

\end{itemize}


\vspace{0.5cm}
\textbf{Acknowledgement}.
This work was supported by NSFC (12001168),
Henan University of Engineering (DKJ2019010), the Key Research Project of Education Department of Henan Province (21A110006), the project ZR2019QA015 supported by Shandong Provincial Natural Science Foundation.

\end{document}